\documentclass[11pt]{amsart}
\numberwithin{equation}{section}

\usepackage{amsmath,amsfonts,amsthm,amsopn,color,amssymb,enumitem}
\usepackage{amsthm,amsmath,amssymb,amscd}
\usepackage{amssymb}
\usepackage{graphics}
\usepackage{color}
\usepackage{esint}

\makeatletter
\@namedef{subjclassname@2020}{\textup{2020} Mathematics Subject Classification}
\makeatother

\newtheorem{theorem}{Theorem}[section]

\newtheorem{proposition}[theorem]{Proposition}

\def\R{\mathbb{R}}
\def\RN{\mathbb{R}^N}

\def\S{\Sigma}

\def\e{\varepsilon}

\begin{document}

\title[Nontrivial solutions to the relative overdetermined torsion problem]{Nontrivial solutions to the relative overdetermined torsion problem in a cylinder}

\author{Filomena Pacella}
\address{(F.~Pacella)
        Dipartimento di Matematica, Università di Roma Sapienza, P.le Aldo Moro 5, 00185 Roma, Italy
} \email{filomena.pacella@uniroma1.it}

\author{David Ruiz}
\address{(D.~Ruiz)
	IMAG, Departamento de An\'alisis matem\'atico, Universidad de Granada,
	Campus Fuente-nueva, 18071 Granada, Spain} \email{daruiz@ugr.es}

\author{Pieralberto Sicbaldi}
\address{(P.~Sicbaldi)
	IMAG, Departamento de An\'alisis matem\'atico,
	Universidad de Granada,
	Campus Fuentenueva,
	18071 Granada,
	Spain \& Aix Marseille Universit\'e - CNRS, Centrale Marseille - I2M, Marseille, France}
\email{pieralberto@ugr.es}

\thanks{{\bf Acknowledgements.} F. P. is member of GNAMPA-INdAM and has been supported by PRIN 2022AKNS E4 - Next Generation EU.
        D. R. has been supported by the FEDER-MINECO Grant PID2021-122122NB-I00 and by J. Andalucia (FQM-116).
	P. S has been supported by the FEDER-MINECO Grant PID2020-117868GB-I00.
	D. R. and P. S. also acknowledge financial support from the Spanish Ministry of Science and Innovation (MICINN), through the \emph{IMAG-Maria de Maeztu} Excellence Grant CEX2020-001105-M/AEI/10.13039/501100011033.}



\maketitle

\noindent

\noindent
\begin{abstract}
	Given a bounded regular domain $\omega \subset \R^{N-1}$ and the half-cylinder $\Sigma = \omega \times (0,+\infty)$, we consider the relative overdetermined torsion problem in $\Sigma$, i.e. 
	\begin{equation*}
		\begin{cases}
			\Delta {u}+1=0 &\mbox{in $\Omega$},\\
			\partial_\eta u = 0 &\mbox{on $\widetilde \Gamma_\Omega$},\\
			u=0 &\mbox{on $\Gamma_\Omega$},\\
			\partial_{\nu}u =c  &\mbox{on $\Gamma_\Omega$}.
		\end{cases}
	\end{equation*}
where $\Omega \subset \Sigma$, $\Gamma_\Omega = \partial \Omega \cap \Sigma$, $\widetilde \Gamma_\Omega = \partial \Omega \setminus \Gamma_\Omega$, $\nu$ is the outer unit normal vector on $\Gamma_\Omega$ and $\eta$ is the outer unit normal vector on $\widetilde \Gamma_\Omega$. We build nontrivial solutions to this problem in domains $\Omega$ that are the hypograph of certain nonconstant functions $v : \overline{\omega} \to (0, + \infty)$.
Such solutions can be reflected with respect to $\omega$, giving nontrivial solutions to the relative overdetermined torsion problem in a cylinder. The proof uses a local bifurcation argument which, quite remarkably, works for any generic base $\omega$.
	
\end{abstract}

\section{Introduction}
\label{Section 1}

If $\Omega$ is a domain of $\mathbb{R}^N$ ($N\geq 2$) the classical overdetermined torsion problem in $\Omega$ is given by:
\begin{equation}\label{serrin}
	\begin{cases}
		\Delta {u}+1=0 &\mbox{in $\Omega$},\\
		u=0 &\mbox{on $\partial \Omega$},\\
		\partial_\nu u = c &\mbox{on $\partial \Omega$},
	\end{cases}
\end{equation}
where $\nu$ is the outer unit normal vector on $\partial \Omega$ and $c$ is a constant. Such overdetermined problem arises when one studies critical points of the torsion energy functional
\[
J(\Omega) = \min_{v \in H_0^1(\Omega)}  \left ( \frac12 \int_{\Omega} |\nabla v|^2 - \int_\Omega v \right)
\]
with respect to volume-preserving deformations of the domain. It is well known that \eqref{serrin} admits a solution if and only if $\Omega$ is a ball, see \cite{serrin}.

\medskip

If we change the ambient space $\mathbb{R}^N$ by a certain open set $\Sigma \subset \mathbb{R}^N$ the expression of the torsion energy functional does not change but now the space $H_0^1(\Omega)$ is naturally substituted by the closure of the space $$\{\varphi\in C^\infty(\Omega), \;\;\textrm{Supp}(\varphi)\subset \Omega\cup\partial \Sigma\}$$ with respect to the $H^1$-norm.
The study of critical points of $J$ with respect to volume-preserving deformations of the domain leads to the {\it relative} overdetermined torsion problem in $\Sigma$ (see Proposition 2.6 in \cite{afonso-iacopetti-pacella-energy}), i.e.
\begin{equation}\label{main}
	\begin{cases}
		\Delta {u}+1=0 &\mbox{in $\Omega$},\\
		\partial_\eta u = 0 &\mbox{on $\widetilde \Gamma_\Omega$},\\
			u=0 &\mbox{on $\Gamma_\Omega$},\\
			\partial_{\nu}u =c  &\mbox{on $\Gamma_\Omega$}.
	\end{cases}
\end{equation}
Here $\Gamma_\Omega = \partial \Omega \cap \Sigma$, $\widetilde \Gamma_\Omega = \partial \Omega \setminus \Gamma_\Omega$, $\nu$ is the outer unit normal vector on $\Gamma_\Omega$ and $\eta$ is the outer unit normal vector on $\widetilde \Gamma_\Omega$. Notice that the overdetermined condition appears only on $\Gamma_\Omega$, that is usually called the {\it relative or free boundary} of $\Omega$, while the zero boundary condition is of mixed type instead of Dirichlet type. Let us point out that 
the existence of a solution to \eqref{main}, under the assumption that $\Gamma_\Omega$ is transversal to $\widetilde \Gamma_\Omega$, forces $\Gamma_\Omega$ to intersect $\partial \Sigma_\omega$ orthogonally, see \cite{lamboley-sicbaldi}.

\medskip

In this paper we consider the relative overdetermined torsion problem in a cylinder. We look for solutions which are even with respect to one variable, so we can restrict ourselves to a half-cylinder. More precisely, given $\omega \subset \mathbb{R}^{N - 1}$ a bounded regular domain, we define: 
\begin{equation}
   \Sigma = \Sigma_\omega = \{(x', x_N) \in \mathbb{R}^{N - 1} \ | \ x' \in \omega, \ x_N \in (0, + \infty)\}.
\end{equation}
A natural question is to characterize domains $\Omega$ for which \eqref{main} admits a solution. For instance, half-balls lying on $\omega$ provide solutions to \eqref{main}, where the solution $u$ is radially symmetric. 
In \cite{afonso-iacopetti-pacella-RLM} some qualitative properties of domains admitting solutions to \eqref{main} were shown, in particular when $\Omega$ is the hypograph of a $C^2$ function $v : \omega \to \mathbb{R}$, i.e. 
\begin{equation*}
    \Omega = \Omega_v = \{(x', x_N) \in \mathbb{R}^N \ | \ x' \in \omega, \ x_N \in (0, v(x'))\}\,.
\end{equation*}
Here $\Gamma_{\Omega_v}$ is just the cartesian graph of $v$ over $\omega$. It is obvious that if $v(x)=t>0$, then \eqref{main} admits a 1D solution \begin{equation}
	\label{eq:U_t}
	U_t(x', x_N) = \frac{t^2 - x_N^2}{2}.
\end{equation}
We will write $U_t(x_N)$ instead of $U_t(x', x_N)$ and refer often to $\Omega_t$ as the trivial domains.
Thus the question is whether rigidity results could be obtained characterizing $\Omega_t$ as the only domains for which \eqref{main} admits a solution, at least for some values of $|\Omega|$ or at least in the class of hypographs. A first step in understanding this question has been made in \cite{afonso-iacopetti-pacella-energy}, where more general relative overdetermined problems in half-cylinders have been considered together with their variational formulation. The stability analysis performed in that paper suggests that nontrivial solutions to \eqref{main} could exist, depending on the first nonzero Neumann eigenvalue of the Laplacian 
on the base $\omega \subset \mathbb R^{N - 1}$ (see Theorem 1.3 in \cite{afonso-iacopetti-pacella-energy}).
The aim of the present paper is to confirm this fact and to show that, indeed, there exist infinitely many hypographs bifurcating from the trivial ones for which \eqref{main} admits a solution.\\

In order to state our result let us consider the eigenvalue problem of the Laplacian operator under Neumann boundary conditions on $\omega \subset \mathbb R^{N - 1}$:
\begin{equation}
    \label{eq:Neumann_eig_prob}
    \begin{cases}
        -\Delta \psi = \lambda \psi & \mbox{in $\omega$},\\
		\partial_\eta \psi = 0 &\mbox{on $\partial \omega$}\,.
    \end{cases}
\end{equation}
Let $\lambda_k$ be its sequence of eigenvalues where
$$
0 = \lambda_0 < \lambda_1 \leq \lambda_2 \leq \ldots
$$
Moreover we denote by $\beta$ the unique solution of the equation:
\begin{equation}\label{beta}
\sqrt{\beta} \tanh(\sqrt{\beta}) = 1,
\end{equation}
which is given by $\beta = 1.43923...$\,.
In this paper we prove the following result:

\begin{theorem} \label{Tmain} Assume that $\lambda_j$ is a simple eigenvalue for some $j \geq 1$. Then there exists a smooth curve:
$$ (-\e, \e) \ni s \mapsto v(s) \in C^{2,\alpha}(\omega)$$ such that:
\begin{enumerate}
	\item $v(0)$ is the constant function $\sqrt{\frac{\beta}{\lambda_j}}$ ;
	\item $v(s)$ is a nonconstant function for $s \neq 0$;
	\item the problem \eqref{main} admits a solution with $\Omega = \Omega_{v(s)}$, being the constant $c$ depending on $s$.
\item $\Gamma_{\Omega_{v(s)}}$ intersects $\partial \Sigma_\omega$ orthogonally.
\end{enumerate} 
\end{theorem}


We point out that, for a generic domain $\omega$, all Neumann eigenvalues are simple (see \cite{henry}). It is important to remark that our result does not depend on the particular cylinder considered and highlights the role played by the Neumann eigenvalues $\lambda_j$ of the base $\omega$, whatever the shape of $\omega$ is. 
We also observe that the solutions to \eqref{main} can be reflected with respect to the base $\omega$ of the half-cylinder (because of the zero Neumann condition) obtaining new solutions to the relative overdetermined torsion problem in a (full) cylinder, with free boundary given by two connected components symmetric with respect to a horizontal hyperplane.

\medskip

The proof of Theorem 1.1 does not use the approach of \cite{afonso-iacopetti-pacella-energy} but relies on a bifurcation analysis in the same spirit of other recent papers building nontrivial solutions to overdetermined problems in unbounded domains of $\mathbb{R}^N$, such as \cite{fall-minlend-weth, S10, SS12, RRS20, RSW22}.

\medskip


In order to put our result in perspective, it is interesting to compare the overdetermined torsion problems with constant mean curvature (CMC) embedded surfaces, that arise (in the compact case) when one studies critical points of the area functional of the boundary of a domain with respect to volume-preserving deformations. Since the seminal paper of Serrin (\cite{serrin}), which used the moving planes method introduced by Alexandroff (\cite{alexandroff}) for the study of CMC surfaces, the two frameworks have been often compared. In the bounded case there is a one-to-one correspondence: the bounded domains where \eqref{serrin} has a solution are balls as the CMC compact surfaces embedded in $\mathbb{R}^n$ are spheres. For unbounded domains in general it is not true that the boundary of a domain where \eqref{serrin} can be solved is a CMC surface. Nevertheless, the domains discovered in \cite{fall-minlend-weth} can be put in correspondence with the CMC Delaunay surfaces because they share a similar shape. A natural question is if there exists a (one-to-one) correspondence between solutions to \eqref{serrin} in unbounded domains and unbounded embedded $n$-dimensional CMC surfaces. Observe that by rescaling CMC surfaces appropriately one obtains a minimal surface and, analogously, by rescaling domains $\Omega$ and solutions $u$ of \eqref{serrin} one arrives to a harmonic function with overdetermined boundary conditions. Quite surprisingly, Traizet proved for $n=2$ a one-to-one correspondence between these two classes of objects, i.e. minimal surfaces and harmonic overdetermined problems (\cite{T13}).

\medskip

In the context of the relative overdetermined torsion problem the natural question is if there is a correspondence between solutions to \eqref{main} and free boundary embedded CMC surfaces in the same ambient space $\Sigma \subset \mathbb{R}^n$, which are embedded CMC surfaces hitting $\partial \Sigma$ orthogonally. These surfaces arise when one studies critical points of the area functional of the free boundary of a domain $\Omega \subset \Sigma$ with respect to volume-preserving deformations.
The result of Theorem \ref{Tmain} provides a negative answer to the question. Indeed, 
it is not difficult to show that embedded free boundary CMC graphs in a cylinder must be flat, (see for instance \cite[Proposition 2.1]{afonso-iacopetti-pacella-RLM}). On the other hand, by virtue of our Theorem \ref{Tmain}, there are one-parameter families of nontrivial free boundary hypographs $\Omega_s \subset \Sigma_w$, $s \in (-\e, \e)$, admitting a solution to \eqref{main}.

\medskip

Finally we compare our result with the relative overdetermined torsion problem in cones studied in \cite{afonso-iacopetti-pacella-energy}, \cite{iacopetti-pacella-weth}, \cite{pacella-tralli}. In this framework a rigidity result can be obtained in convex cones showing that the only domains admitting solutions to such problem are spherical sectors, see Theorem 1.1 in \cite{pacella-tralli}. Parallel rigidity results can be obtained for free boundary CMC surfaces inside a cone, see Theorem 1.2 in \cite{pacella-tralli}. On the other hand, at least for the relative overdetermined torsion problem, the symmetry result breaks down for a class of nonconvex cones, (see \cite{iacopetti-pacella-weth}). 
Instead, in the case of the half-cylinder, the shape of the domain $\omega$, and in particular its convexity, does not prevent the existence of nontrivial domains for the relative overdetermined torsion problem. 

\medskip

The paper is organized as follows. In Section \ref{sec:rel_dic_prob} we define and study a normal derivative operator associated to each hypograph. In Section \ref{sec:linearization} and Section \ref{sec:linearized_operator_H_t} we analyze its linearization and study the related eigenvalues. In Section \ref{sec:proof_Tmain} we prove Theorem \ref{Tmain} by means of the Crandall-Rabinowitz Theorem.

\section{The relative Dirichlet problem and normal derivative operator}
\label{sec:rel_dic_prob}

Let $\omega$ be a bounded and regular domain of $\R^{N-1}$. Let us consider the half-cylinder $\S_\omega \subset \RN$ with base $\omega$, i.e.
\[
\S_\omega = \{\, (x', x_N) \in \R^{N-1} \times \R \, |\, x' \in \omega\,, \, x_N \in (0, +\infty)\, \}\,.
\]
Given a continuous function $v: \omega \rightarrow(0,\infty),$ we 
denote
\[
\Omega_v=\left\{(x',x_N)\in \S_{\omega} \,|\, x' \in \omega\, , \, 0 < x_N<v(x') \right\}.
\]
We write the boundary $\partial \Omega_v$ as $\Gamma_v \cup \tilde  \Gamma_v$, where $\Gamma_v = \partial \Omega_v \cap \S_{\omega}$ and $\tilde \Gamma_v = \partial \Omega_v \cap \partial \S_{\omega}$. { In this domain we will also denote by $\eta$ the unit normal vector on $\tilde{\Gamma}_v$; instead, $\nu$ stands for the unit normal vector on $\Gamma_v$. Let us point out that on $\partial \Gamma_v$ both normal vectors are different, and they are orthogonal if $\partial_{\eta} v =0$ on $\partial \omega$.}  
In the case $v \equiv t \in \mathbb{R}$ we use the notations $\Omega_t$, $\Gamma_t$, $\widetilde \Gamma_t$ and in the special case $t = 1$ we simply write $\Omega_1 = \Omega$, $\Gamma_1 = \Gamma$, $\widetilde \Gamma_1 = \widetilde \Gamma$.

Fix $\alpha \in (0,1)$.  For $k \in \mathbb{N} \setminus \{0\}$, define:
\begin{equation} \label{space} X_k= \{ v \in C^{k,\alpha}(\omega):\  \partial_{\eta} v=0 \mbox{ on }\partial \omega\}, \end{equation} 
\begin{equation} \label{space2} \tilde{X}_k= \{ v \in X_k: \ \int_{\omega} v=0\}, \end{equation} endowed with the usual $C^{k, \alpha}$ norm, denoted as $\| \cdot \|_{C^{k,\alpha}}$.
Clearly, any function $v \in X_k$ can be decomposed as $v = \tilde{v} + h$, where $\tilde{v} \in \tilde{X}_k$, $h \in \R$. Indeed,
$$h = \frac{1}{|\omega|} \int_{\omega} v.$$

For any $v \in X_k$ we consider the problem:
\begin{equation}\label{dir}
	\begin{cases}
		\Delta {u}+1=0 &\mbox{in $\Omega_{v},$}\\
		u=0 &\mbox{on $\Gamma_{v},$}\\
		\partial_\eta u = 0 &\mbox{on $\tilde \Gamma_{v},$}
	\end{cases}
\end{equation}
This problem admits a unique weak solution $u_v$ in the space $H_0^1(\Omega_v \cup \widetilde \Gamma_v)$ which is the Sobolev space of the functions in $H^1(\Omega_v)$ whose trace vanishes on $\Gamma_v$. This comes by a standard application of the Lax-Milgram theorem. Moreover by the orthogonal conditions on the intersection between $\Gamma_v$ and $\partial \Sigma_\omega$ (implied by \eqref{space}) and by the Neumann condition on $\widetilde \Gamma_v$, we have that the $u_v$ is actually a $C^{2, \alpha}(\overline \Omega_v)$ solution of \eqref{dir} (see for example \cite{lamboley-sicbaldi, pacella-tralli}).

In the next proposition we want to show that, roughly speaking, the function $u_{v}$ depends on $v$ in a smooth way. In order to make this assertion rigorous, we use the following diffeomorphism  $Y_v: \Omega \to \Omega_{v}$,

\begin{equation} \label{Y}
	{ Y_v (x', x_N) : = \left(x',   v (x') x_N \right) },
\end{equation}
valid for any $v \in X_k$, $v>0$. 
The coordinates we consider from now on are $(x',x_N) \in \Omega$ and in these coordinates the metric  $g_v$ given by the pullback of the Euclidean metric by $Y_v$ can be written as
	\begin{equation} \label{metrica}
	g_v = d(x')^2 + x_N^2 dv^2 + 2v(x') x_Ndv dx_N + v(x')^2 dx_N^2\,.
	\end{equation}
	being $dv=\frac{\partial v}{\partial x_1} dx_1+ \cdots +\frac{\partial v}{\partial x_{N-1}}dx_{N-1}$ and $d(x') = (dx_1, ..., dx_{N-1})$.
Then, $u_v$ is a solution of problem \eqref{dir} if and only if $\phi_v= Y_v^* u$ solves:
\begin{equation}
	\label{formula-new}
	\left\{
	\begin{array}{rccl}
		L_v(\phi) + 1 & = & 0 & \mbox{on $\Omega$},\\
		 \phi& = & 0 &\mbox{on $\Gamma$},\\
		\partial_{ \eta}  \phi & = & 0 &\mbox{on $\tilde \Gamma$},
	\end{array}
	\right.
\end{equation}
where $L_v= \Delta_{g_v}$ is the Laplace-Beltrami operator with respect to the metric $g_v$. If we denote
	\[
	h(x') = \frac{1}{v(x')},
	\]
	we have
	\begin{eqnarray*}
		L_v (\phi) & = & \sum_{i=1}^{N-1} \left[ \phi_{ii} + \frac{x_N}{h}\, \left( 2 \phi_{iN}\, h_i  + \phi_{N}\, h_{ii} + \phi_{NN}\, (h_i)^2\, \frac{x_N}{h} + \phi_N\, (h_i)^2\right) \right] + \phi_{NN}\, h^2.
	\end{eqnarray*}
where subindices stand for partial derivatives.

For later use it is convenient to define the space:
$$ Z = \{ \phi \in C^{2, \alpha}(\Omega): \ \phi=0 \mbox{ on } \Gamma, \ \partial_\eta \phi=0 \mbox{ on } \tilde{\Gamma} \}.$$

%

\begin{proposition}
\label{dirichlet}
For all $v \in X_2$, $v>0$, there exists a unique positive solution $\phi_v \in {C}^{2, \alpha} (\Omega)$ to the problem \eqref{formula-new} (equivalently, there exists a unique positive solution $u_v$ of \eqref{dir}). In addition the map:
\begin{equation} \label{defPhi} \phi: \{ v \in X_2, \ v >0\} \to Z,  \ \phi(v)= \phi_v \end{equation}
is $C^2$.
\end{proposition}
\begin{proof} As commented previously, there exists a unique solution to \eqref{dir}. Via the diffeomorphism $Y_v$ we conclude existence and uniqueness for \eqref{formula-new}. Hence we focus on the $C^2$ dependence.

%



We define $N: \{v \in X_{2}, \ v>0\} \times Z \to C^{0,\alpha}(\Omega)$,
\begin{equation}\label{N}
N (v , \phi) : =  L_{v} (\phi) + 1 \,, 
\end{equation}
Clearly $N$ is a smooth map. Consider now $v_0 \in {X_2}$, $v_0 >0$, and $\phi_0= \phi_{v_0}$ a solution of \eqref{formula-new}, so that
\[
N (v_0, \phi_0) =0.
\]
We can compute the partial derivative with respect to the variable in $Z$ at the point $(v_0,\phi_0)$: 
\[
D  N_{(v_0,\phi_0)}{(0, \psi)}  = L_{v_{0}}(\psi).
\] 
We now observe that $L_{v_0}$ is an isomorphism from $Z$ to $C^{0,\alpha}(\Omega)$. Indeed, composing with the diffeomorphism $Y$, this is due to the fact that for any $f \in C^{0,\alpha}(\Omega_v)$, the problem:
\begin{equation*}
	\begin{cases}
		\Delta {\phi}= f &\mbox{in $\Omega_{v},$}\\
		u=0 &\mbox{on $\Gamma_{v},$}\\
		\partial_\eta u = 0 &\mbox{on $\tilde \Gamma_{v},$}
	\end{cases}
\end{equation*}
admits a unique solution in $C^{2, \alpha}(\Omega_v)$. This follows from the Lax-Milgram theorem and regularity theory as for \eqref{dir}.
Therefore, the Implicit Function Theorem implies that, for any $v$ in a neighborhood of $v_0$ in $X$, there exists a unique solution $\phi_v$ in a neighborhood of $\psi_0$ in $Z$ so that
\[
N(v, \phi_v)=0\,.
\] 
Clearly $\phi_v = \phi(v)$ and this is a $C^2$ map, as claimed.
\end{proof}

The previous proposition allows us to define the operator: 
\[
F: \{v \in X_2:\ v >0 \} \to \tilde{X}_1
\] 
defined by
\begin{equation}\label{F}
F (v) (x') =  \displaystyle  \frac{\partial u_v}{\partial {\nu}} (x', v(x')) - \frac{1}{|\Gamma_v|} \, \int_{\Gamma_v} \, \frac{\partial u_v}{\partial \nu}
\end{equation}
where $|\Gamma_v|$ is the $(N-1)$-dimensional measure of $\Gamma_v$ and $u_v$ is the solution of \eqref{dir} provided by Proposition \ref{dirichlet}. Observe that $F$ is well defined since on $\partial \Gamma_v$,
$$   \frac{\partial}{\partial \eta} \left (\frac{\partial u_v}{\partial \nu} \right )=\frac{\partial}{\partial \nu} \left (\frac{\partial u_v}{\partial \eta} \right )=0.$$
Clearly the zeros of $F$ correspond to the solutions of \eqref{main} in $\Omega_v$. We plan to find nontrivial zeroes of $F$ as a local bifurcation of a family of trivial solutions. Indeed, if we take $v {\equiv} t $ a positive constant, the unique solution of \eqref{dir} is given by:
	\begin{equation} \label{defU} U_t(x',x_N) = \frac{t^2-x_N^2}{2}. \end{equation}
We denote $U=U_1$ and morever we will write $U_t(x_N)$ instead of $U_t(x',x_N)$. 
Clearly the Neumann derivative is constant so that $F(t) =0 $ for any constant $t>0$. 

It is easy to see that the operator $F$ has the following equivalent expression: 
 \[
{ F (v) (x') = - \left ( \displaystyle  \left|\nabla_{g_v} \phi_v \right|_v (x', 1) - \frac{1}{m_v} \, \int_{\omega \times \{1\} } \,  \left|\nabla_{g_v} \phi_v \right|  dv_g \right )}.
 \]
Here we denote $| \cdot |_v$ as the norm in the metric $g_v$, and $m_v$ the $N-1$ dimensional measure of $\omega \times \{1\}$ in the metric $g_v$. Observe that:
$$ \left|\nabla_{g_v} \phi_v \right|_v (x', 1) = \sqrt{\sum_{i=1}^{N-1} \left ( \partial_{x_i} \phi_v - \frac{\partial_{x_N} \phi_v \, \partial_{x_i} v}{v(x')^2} \right )^2 + \left (\frac{\partial_{x_N} \phi_v}{v(x')^2} \right )^2}.$$
By the Hopf principle the above term is always different from zero. Using this alternative expression it is clear that $F$ is a $C^2$ operator.

\section{The linearization of the normal derivative operator}
\label{sec:linearization}

We will next compute the Fr\'{e}chet derivative of the operator $F$ at any constant function $v \equiv t$ with respect to variations in $\tilde{X}_2$. We denote this derivative by $DF_t$. In order to do this we introduce the following linear non-homogeneous problem. For any positive constant $t$ and any $v\in X_2$, there exists a unique solution $\psi_{t,v}$ to the problem
\begin{equation}\label{psiv}
  \begin{cases}
  \Delta\psi_{t,v}=0 &\emph{in $\Omega_{t}$, }\\
  \psi_{t,v}=  v &\emph{on $\Gamma_{t}$.}\\
 \partial_{\eta} \psi_{t,v} = 0 &\emph{on $\tilde \Gamma_{t}$.}\\
  \end{cases}
\end{equation}
where we identified $\Gamma_{t}$ with $\omega$ in order to write the second equality. We remark that the normal derivative of $\psi_{t,v}$ on $\Gamma_{t}$, i.e. $\left. \partial_\nu \psi_{t, v}\right|_{\Gamma_{t}}$, is nothing else than $\partial_{y_N} \psi_{t,v} (y',t)$, i.e. it is a function of the variable $y' \in \omega$ with $C^{1,\alpha}$ regularity. Notice that $\left.\partial_{\nu} \psi_{t,v}\right|_{\Gamma_{t}}$ has 0 mean, as one can see integrating the first equation of \eqref{psiv}. Moreover, this function has 0 normal derivative at $\partial \omega$, since:
$$   \frac{\partial}{\partial \eta} \left (\frac{\partial \psi_{t,v}}{\partial \nu} \right )= \frac{\partial}{\partial \nu} \left (\frac{\partial \psi_{t,v}}{\partial \eta} \right )= 0.$$
Summarizing, $\left.\partial_{\nu} \psi_{t,v} \right|_{\Gamma_{t}}$ belongs to the space $\tilde{X}_1$. 
We are now in position to compute, for every $t > 0$, the Fr\'{e}chet derivative $DF_t$ of the operator $F$ on the space $\tilde{X}_2$ (recall \eqref{space2}) that is, on the directions which are orthogonal to the constant functions.

\begin{proposition} \label{Pr31}
For any $t>0$ and any $w \in \tilde{X}_2$
\[
D F_{t}(w)= H_{t}(w)\, ,
\]
where { $H_{t}: \tilde{X}_2\rightarrow \tilde{X}_1$ is the linear continuous operator defined by
\begin{equation}\label{eq25}
H_{t}(w)(x') = t \partial_{\nu} (\psi_{t,w})(x',t) - w(x') .
\end{equation}}
 \end{proposition}

\begin{proof}
	By the $C^1$ regularity of $F$, it is enough to compute the linear operator obtained by the directional derivative of $F$ with respect to $w \in \tilde{X}_2$, computed at $t$. Such derivative is given by
\begin{align*}
  \mathop {\lim}\limits_{s\rightarrow 0}\frac{F(t + sw)-F(t)}{s}=\mathop {\lim}\limits_{s\rightarrow 0}\frac{F(t+sw)}{s}.
    \end{align*}
    Fix $t >0$, let $v=t+sw$ and recall the diffeomorphism $Y_v$ defined in \eqref{Y}. The function $\phi_v= Y_v^* u_v$ solves the problem
    \begin{equation*}
     \begin{cases}
     L_{v} \phi +1=0 &\mbox{in $\Omega$},\\
      \phi=0 &\mbox{on $\Gamma$},\\
\partial_{\eta} {\phi} =0 &\mbox{on $\tilde \Gamma$}.
      \end{cases}
    \end{equation*}

Recall now the definition of $U_t$ given in \eqref{defU} and define $U_{t,v}= Y_v^* U_t$, which has the explicit expression:
\begin{equation} \label{bordo} U_{t,v}(x',x_N)= \frac{t^2 - x_N^2 v(x')^2}{2}  = \frac{t^2 - x_N^2(t+s\, w(x'))^2}{2}. \end{equation}
Clearly, 
\[L_{v}  U_{t,v} + 1=0, \quad \text{ in } \Omega.\]
Let $\hat{\psi}= \phi_v- U_{t,v} $, which is a solution of:
\begin{equation}\label{eq37}
     \begin{cases}
     L_v \hat{\psi}=0  &\mbox{in $\Omega$},\\
      \hat{\psi}=- U_{t,v} &\mbox{on $\Gamma$},\\
      \partial_{\eta}  \hat{\psi}=0  &\mbox{on $\widetilde \Gamma$}.\\
      \end{cases}
    \end{equation}
 Obviously, $\hat{\psi}$ is differentiable with respect to $s$. When $s=0$, we have $\hat{\psi}=0$: we set
  \[\dot{\psi}=\partial_{s}\hat{\psi}|_{s=0}.\]
 Differentiating (\ref{eq37}) with respect to $s$ at  $s=0$ and taking into account \eqref{bordo}, we have
  \begin{align*}
     \begin{cases}
     L_t \dot{\psi}=0  &\mbox{in $\Omega$, }\\
      \dot{\psi}=\, t {w}  &\mbox{on $\Gamma$,}\\
     \partial_{\eta} \dot{\psi}=0  &\mbox{on $\tilde \Gamma$,}\\
  \end{cases}
    \end{align*}
    In other words,
    \begin{equation} \label{sanremo} \dot{\psi}(x',x_N)= t \, ( \psi_{t,w}\circ Y_t) (x', x_N) = t \, \psi_{t,w}(x', t\, x_N)\end{equation}  where $\psi_{t, w}$ is defined by \eqref{psiv}. Then, we can write
\[\phi_{v}={U}_{t,v} +s \dot{\psi} +\mathcal{O}(s^{2}).\]
    Taking into account the expression \eqref{bordo}, we have
   \begin{align*}
   	\phi_{v}(x',x_N)={U}_{t,v}(x',x_N) +s \dot{\psi}(x',x_N) +\mathcal{O}(s^{2}) \\ = t^2 \, \frac{ 1- x_N^2}{2} +s \big(-x_N^2 t \, w(x') +\dot{\psi}(x',x_N) \big) +\mathcal{O}(s^{2}).
         \end{align*}
         In order to complete the proof of the result, we need now to compute the normal derivative of the function $\phi_{v}$ with respect to the metric $g_v$, that we can write as $g_v(\nabla^{g_v} \phi_v, \nu_v )$. Taking into acount \eqref{metrica}, the metric $g_v$ can be expanded as:
         \[g_v= d(x')^2+ s^2\, x_N^2\, d\, w^2 +2 s\, x_N\,(t+s{w})\,d w\, dx_N+(t+s w)^{2}\, d(x_N)^2\,.
         \]
        It follows from this expression that the unit normal vector $\nu_v$ for the metric $g_v$ is given by
          \[
          \nu_v= \mathcal{O}(s) \sum_{i=1}^{N-1} \partial_{x_i} + \left( (t+s w)^{-1}+\mathcal{O}(s^{2})\right)\partial_{x_N} =  \mathcal{O}(s) \sum_{i=1}^{N-1} \partial_{x_i} + \left(\frac 1 t - s \frac{w}{t^2} +\mathcal{O}(s^{2})\right)\partial_{x_N}\,.
          \]
   By this expression and taking into account \eqref{sanremo} we conclude that if $x_N =1$,
   \begin{eqnarray*}
   g_v(\nabla^{g_v} \phi_v, \nu_v )& =& \left(\frac 1 t - \frac{ s w}{t^2} \right) (-t^2)+  \frac s t \big(\partial_{x_N}{ {\dot{\psi}}} -2 t w \big )+\mathcal{O}(s^{2})\\
   & = & {-t +s\left( \frac 1 t \partial_{x_N}{ {\dot{\psi}}} -w  \right)+\mathcal{O}(s^{2})} \\ & = & -t +s\big( t \, \partial_{x_N}{ {\psi}_{t,w} \circ Y_t} -w  \big)+\mathcal{O}(s^{2}),
   \end{eqnarray*}
 which concludes the proof, since on $\Gamma_v$ we have $\frac{\partial u_v}{\partial {\nu}} (x', v(x')) = g_v(\nabla^{g_v} \phi_v, \nu_v )$, and $\left.{\psi}_{t,w}\right|_{\Gamma_v}$ and $w$ have mean zero.
  \end{proof}

\section{Study of the linearized operator $H_t$}
\label{sec:linearized_operator_H_t}

In this section we study the properties of the operator $H_t$ defined in \eqref{eq25}. On that purpose, take $\{ \xi_k\}_{k\geq 0}$ a sequence of eigenfunctions of the Laplacian operator on $\omega$ with Neumann boundary conditions. We denote $\lambda_k$ the corresponding sequence of eigenvalues, where:
$$ 0 = \lambda_0 < \lambda_1  \leq \lambda_2 \dots. $$

\begin{proposition} \label{H} For any $t>0$ the operator $H_t$ is an essentially self-adjoint Fredholm operator of index 0. Moreover, its eigenvalues are given by:
\begin{equation} \label{eigenvalues} \mu_{t,k}= t \sqrt{\lambda_k} \tanh( t \sqrt{\lambda_k}) -1, \  k \geq 1.\end{equation}
The associated eigenfunctions are just {$\{ \xi_k\}_{k\geq 1}$.}
\end{proposition}

\begin{proof}

Let us fix $t>0$ in the following discussion. In order to prove the essential self-adjointness of $H_t$ we will extend its definition to a natural Hilbert space. We define the following spaces:
$$
\widetilde H^1(\omega) = \left\{v \in H^1(\omega) \ : \ \int_\omega v = 0 \right\}, \quad \widetilde L^2(\omega) = \left\{v \in L^2(\omega) \ : \ \int_\omega v = 0\right\}.
$$
Then any function $v$ in $\widetilde H^1(\omega)$ can be written as 
$$ v(x')= \sum_{k=1}^{+\infty} a_k \, \xi_k \left (x' \right )$$
where $a_k$ are the Fourier coefficients. 

This, in particular, holds for $v \in \widetilde X_2$. Then, $\psi_{t,v}$ is given by:
$$ \psi_{t,v}(x',x_N)= \sum_{k=1}^{+\infty} a_k \, g_k(x_N) \, \xi_k (x'),$$
for some functions $g_k$. Observe that:
$$ \Delta \psi_{t,v}(x',x_N) = \sum_{k = 1}^{+ \infty}a_kg_k''(x_N) \xi_k(x') + a_kg_k(x_N)(-\lambda_k \xi_k(x')).$$
Plugging this expression in \eqref{psiv}, we have that the functions $g_k$ satisfy the equation:
\begin{equation}\label{g_k}
	\begin{cases} g_k''(x_N) - \lambda_k g_k(x_N)=0, & x_N \in (0,t), \\
		g_k'(0)= 0, \ g_k(t)= 1. &
	\end{cases}
\end{equation}
Clearly, $$g_k(x_N)= \frac{\cosh(\sqrt{\lambda_k}\, x_N)}{\cosh( \sqrt{\lambda_k}t)}.$$ 
Then,
$$ \partial_{\nu} \psi_{t,w}= \sum_{k=1}^{+\infty} a_k \, g_k'(t) \, \xi_k\left (x') \right )= \sum_{k=1}^{+\infty} a_k \, \sqrt{\lambda_k} \tanh( t \sqrt{\lambda_k}) \xi_k\left (x' \right ).$$
Therefore we can write $H_t(v)$ as:
\begin{equation} \label{HFourier} H_t(v)= t \, \partial_{\nu} \psi_{t,w} - w = \sum_{k=1}^{+\infty} a_k \, \left ( t \sqrt{\lambda_k} \tanh( t \sqrt{\lambda_k}) -1 \right ) \xi_k\,. \end{equation}
This expression gives the characterization of the eigenvalues and eigenfunctions of $H_t$.
Observe that we have the asymptotic behavior:
$$  t \sqrt{\lambda_k} \tanh( t \sqrt{\lambda_k}) \sim t \sqrt{\lambda_k} \ \mbox{ as } k \to + \infty.$$
From this we deduce that $H_t$ can be extended to the Sobolev framework:
$$ H_t: \tilde{H}^1(\omega) \to \tilde{L}^2(\omega),$$
where $H_t$ is a self-adjoint operator of order 1. Moreover, $H_t + Id$ is invertible and, by the compact embedding, we have that $(H_t+ Id)^{-1}: \tilde{L}^2(\omega) \to \tilde{L}^2(\omega)$ is a compact operator. By H\"{o}lder regularity, $H_t$ is a Fredholm operator of index 0. 
Finally, from the expression \eqref{HFourier} we conclude that the eigenvalues of $H_t$ are given by \eqref{eigenvalues}.
\end{proof}

\section{Bifurcation argument and proof of the main result}
\label{sec:proof_Tmain}

In this section we conclude the proof of our main result by means of the classical Crandall-Rabinowitz Theorem. For convenience of the reader we recall its statement below, in a $C^2$ version.

\begin{theorem} \label{CR} \textbf{\mbox{(Crandall-Rabinowitz Bifurcation Theorem)}}
	Let $X$ and $Y$ be Banach spaces, and let $U\subset X$ and $I\subset\mathbb{R}$ be open domains, where we assume $0\in U$. Denote the elements of $U$ by $v$ and the elements of $I$ by $t$. Let $G:I \times U \rightarrow Y$ be a $C^{2}$ operator such that
	\begin{itemize}
		\item[i)] $G(t,0)=0$ for all $t\in I,$
		\item[ii)] $\emph{Ker~} D_{v}G(t_{*},0)=\mathbb{R}\,w$ for some $t_{*}\in I$ and some $w\in X\setminus\{0\};$
		\item[iii)] $\emph{codim Im~} D_{v}G(t_{*},0)=1;$
		\item[iv)] $ D_{t}D_{v}G(t_{*},0)(w)\notin \emph{Im~} D_{v}G(t_{*},0).$
	\end{itemize}
	Then there exists a nontrivial $C^1$ curve
	\begin{equation}\label{eq501}
		(-\e, \e) \ni s \mapsto (t(s), v(s)) \in I \times X,
	\end{equation}
for some $\e>0$, such that:
\begin{enumerate}
	\item $t(0)=t_*$, $t'(0)=0$, $v(0)=0$, $v'(0)=w$.
	\item $G\big(t(s),v(s)\big)=0$ for all $ s\in(-\delta,+\delta)$.
\end{enumerate}  
	Moreover, there exists a neighborhood $\mathcal{N}$ of $(t_{*},0)$ in $X \times \Gamma$ such that all solutions of the equation $G(t,v)=0$ in $\mathcal{N}$ belong to the trivial solution line $\{(t,0)\}$ or to the curve (\ref{eq501}). The intersection $(t_{*},0)$ is called a bifurcation point.
\end{theorem}

We are now in a position to prove Theorem \ref{Tmain}. For this we only need to apply Theorem \ref{CR} to our operator $F$. On that purpose we decompose
$$ X_2 = \R \oplus \tilde{X}_2.$$

 Clearly, i) is satisfied. 
 
 If $\lambda_j$ is simple, then we can take $t_*=\sqrt{\frac{\beta}{\lambda_j}}$ where $\beta$ is given by \eqref{beta}, and then the kernel of $H_{t^*}$ is spanned by $\xi_j$. So ii) is satisfied with $w= \xi_j$. 
 
 Since $H_t$ is essentially self-adjoint, we have that:
$$ \emph{Im~}D_vF(t_*,0) = \{z \in X_1:\ \int_{\omega} z \xi_j =0\},$$
which has codimension 1, as required by iii).

It remains to check condition iv), i.e. the so-called {\it transversality condition}. By Proposition \ref{H} we have that:
$$D_{v}F(t,0)(\xi_j) = (t \sqrt{\lambda_j} \tanh( t \sqrt{\lambda_j}) -1) \xi_j = f(t \sqrt{\lambda_j} ) \xi_j,$$
where $f(s) = s \tanh(s) - 1$. 
We not compute the derivative of the above expression in $t=t_*$ to obtain:
$$D_t D_{v}F(t_*,0)(\xi_j) =  \sqrt{\lambda_j} f'(t_* \sqrt{\lambda_j}) \xi_j= \sqrt{\lambda_j} f'(\sqrt{\beta}) \xi_j.$$ Since $f'(\sqrt{\beta})>0$ (indeed $f'(s)>0$ for any $s>0$) we conclude that $D_t D_{v}F(t_*,0)(\xi_j)$
does not belong to $ \emph{Im~} D_vF(t_*,0)$.

\bibliographystyle{amsplain}

\end{document}